\documentclass{journal}
\usepackage{tikz}
\usepackage{color}
\usepackage[all]{xy}
\usepackage{amsmath, amssymb}
\usepackage{amsfonts}
\usepackage{mathrsfs}
\usepackage{amsthm}
\usepackage{bm}
\usepackage{amscd}
\usepackage{tikz-cd}

\makeatletter
 \def\th@plain{\upshape}
 \makeatother

\newtheorem{theorem}{Theorem}[section]
\newtheorem{proposition}[theorem]{Proposition}
\newtheorem{lemma}[theorem]{Lemma}

\newtheorem{definition}[theorem]{Definition}
\newtheorem{rem}[theorem]{Remark}
\newtheorem{question}[theorem]{Question}

\title{An embedding of spherical quandles into Lie groups}
\author{
Ayu Suzuki
\thanks{Division of Mathematical and Physical Science, Graduate School of Science, Japan Women's University, 2-8-1 Mejirodai, Bunkyo-ku Tokyo, 112-8681, Japan, E-mail: \texttt{m2016044sa@ug.jwu.ac.jp}}
\and
Kentaro Yonemura
\thanks{Osaka Central Advanced Mathematical Institute, Osaka Metropolitan University, Sugimoto, Sumiyoshi-ku, Osaka City 558-8585, Japan, E-mail: \texttt{yonemura-kentaro@sei.co.jp}}
}
\date{}

\usepackage{lineno}
\begin{document}

\maketitle
\begin{abstract}
We construct smooth embeddings of spherical quandles into conjugation quandles of Lie groups, where the ambient Lie groups can be chosen to be orthogonal, Spin, or Pin groups. Moreover, in dimensions $1$ and $3$, we compare our embeddings with those due to Bergman and Akita.
\end{abstract}



\section{Introduction}
A \textit{quandle} is an algebraic structure closely related to knot theory and symmetric space theory. Joyce \cite{Joyce1982}, who coined the term quandle, noted in his introduction that symmetric spaces provide canonical examples of quandles; namely, the quandle operation is defined via the geodesic symmetry at each point. Ishikawa \cite{Ishikawa} introduced the notion of a smooth quandle, which generalizes quandles arising from symmetric spaces, and developed its basic theory. A smooth quandle is a differentiable manifold with a smooth operation satisfying the quandle axioms. Quandles are often compared to groups, and smooth quandles are analogous to Lie groups.

Embedding quandles into conjugation quandles of groups is an important problem and has been studied extensively. Joyce \cite{Joyce1982} gave an embedding of free quandles and suggested that the theory of quandles may be viewed as a generalization of conjugation in groups. Joyce introduced conjugation quandles as a fundamental class of examples; see \cite[Section 4]{Joyce1982} for details. An important early result related to this problem is due to Ryder~\cite{ryder1996algebraic}, who proved that the natural map from the fundamental quandle of a link to the fundamental group of its complement is injective if and only if the link is prime. Eisermann \cite{Eisermann2014} defined a \textit{quandle covering} and developed the basic theory. In his paper, he pointed out that the covering theory of quandles embedded into groups is essentially the theory of central group extensions. Bardakov, Dey, and Singh \cite{BardakovDeySingh2017} dealt with the following problem: 
\begin{question}[\cite{BardakovDeySingh2017}, Question 3.1]
\label{Question_embedding-problem}
For which quandles $X$ does there exist a group $G$ such that $X$ embeds in the conjugation quandle $\operatorname{Conj}(G)$?
\end{question}
Bardakov et al. \cite{BardakovDeySingh2017} and subsequent authors exhibited quandles for which the answer to Question \ref{Question_embedding-problem} is positive. See \cite{Akita2022embedding} and its references for more details. In connection with Question \ref{Question_embedding-problem}, we propose the following question.

\begin{question}
\label{conj_embedding_of_smooth_quandles}
Let $X$ be a topologically connected and algebraically connected smooth quandle. Does there exist a Lie group $G$ and a smooth embedding $\iota:X\to G$ such that $\iota$ is a quandle homomorphism into the conjugation quandle \(\operatorname{Conj}G\)?
\end{question}
Question~\ref{conj_embedding_of_smooth_quandles} asserts the existence of an embedding that is simultaneously a smooth embedding of manifolds and a quandle homomorphism. The question may be regarded as a smooth analogue of Question \ref{Question_embedding-problem}. This formulation is useful because it allows us to use not only algebraic conditions but also geometric conditions. For example, in studying Question~\ref{conj_embedding_of_smooth_quandles}, we often use the universal covering $p:\widetilde{G}\to G$ of a connected Lie group $G$. The map $p$ can be viewed in three ways: as a topological covering, as a quandle covering defined by Eisermann~\cite{Eisermann2014}, and as a central extension. This perspective will be used in the proof of Theorem \ref{main_theorem}. Even if Question~\ref{conj_embedding_of_smooth_quandles} admits counterexamples, it is still important to identify classes of quandles for which Question~\ref{conj_embedding_of_smooth_quandles} holds.

There are infinitely many examples of quandles for which Question~\ref{conj_embedding_of_smooth_quandles} has an affirmative answer. For instance, if we equip an arbitrary smooth manifold with the trivial quandle structure, then the Whitney embedding~\cite{whitney1936differentiable,whitney1944self} into a Euclidean space, viewed as an abelian Lie group, yields a smooth quandle embedding.

In this paper, we answer the question affirmatively for spherical quandles $S^{n}_{\mathbb{R}}$, that is, the quandles on the sphere induced by its structure as a Riemannian symmetric space.

\begin{theorem}
\label{main_theorem}
For any positive integer $n$, there is a Lie group $G_{n}$ and a smooth embedding $\iota_{n}:S^{n}\to G_{n}$ such that $\iota_n$ is an injective quandle homomorphism from the spherical quandle $S^n_{\mathbb{R}}$ to the conjugation quandle $\operatorname{Conj}G_n$. Moreover, $G_n$ may be chosen as follows:
\[
G_{n}
=
\left\{
\begin{array}{ll}
O(2)&  (n=1)\\
Spin(n+1)& (n \mbox{ is even})\\
Pin^{\pm}(n+1)& (n\mbox{ is odd and }n\geq3)
\end{array}
\right.
.
\]
\end{theorem}

Our motivation for proving Theorem \ref{main_theorem} is to construct an example of a non-faithful quandle that embeds into a conjugation quandle. For any quandle, there is a natural quandle homomorphism from it to its inner automorphism group. A quandle is \textit{faithful} if the natural quandle homomorphism is injective.  Quandles arising from Riemannian symmetric spaces are Joyce's model cases of quandles. In this case, the faithfulness of the quandles is related to the Cartan embedding. See also \cite[Example B.3]{NSK}. One often asks whether a given quandle is faithful when we approach Question \ref{Question_embedding-problem}. If a given quandle is not faithful, we need to find a group whose conjugation quandle contains the given quandle. The quandle treated in Theorem \ref{main_theorem} is not faithful. In general, few methods to find an appropriate group are known, and this is why Question \ref{Question_embedding-problem} is difficult. Spherical quandles, which we treat in this paper, are not faithful. However, we consider the universal covering group of the inner automorphism group of spherical quandles and obtain a group into which the quandle embeds. The manifold structure of spherical quandles is a key ingredient in the proof.

We compare the low-dimensional cases of our results with Akita's embedding \cite{Akita2022embedding} of twisted conjugation quandles and with Bergman's embedding \cite{Bergman2021core} of core quandles. Although the work of Akita and Bergman provides embeddings in a fairly general setting, it is rooted in an algebraic perspective and is primarily motivated by applications to finite or, more generally, discrete quandles. In contrast, the present paper focuses on quandle structures arising from symmetric spaces, in particular from spheres, and aims to clarify how these fit into and relate to the embedding theories of Akita and Bergman. Moreover, the objects appearing in Theorem \ref{main_theorem} involve abstract groups such as the Pin and Spin groups, as well as the associated quandles; by exploiting low-dimensional exceptional isomorphisms, we describe them in terms of more familiar Lie groups such as \(SO(2)\) and \(SU(2)\), thereby placing our discussion in a more classical algebraic framework. Our goal is not to construct and compare embeddings on a case-by-case basis, but rather to extract from our results a unified framework for understanding these phenomena in future work. Building on this viewpoint, the first author \cite{AyuSuzuki} observed that spherical quandles and core quandles are homogeneous quandles. She provided a sufficient condition for the existence of an embedding from a homogeneous quandle into a conjugation quandle. In particular, both the embedding and the target quandle can be described explicitly. Using this, she extended Theorem~\ref{main_theorem} to, for example, the case of oriented real Grassmann manifolds.

The present paper incorporates material from the master's thesis \cite{AyuSuzuki2026masterthesis} and the doctoral dissertation \cite{Yonemura2023embedding}.

This paper is organized as follows. In Section~\ref{section_preliminaries}, we recall basic notation and facts on quandles. In Section~\ref{chapter_Action_of_Lie_groups_on_smooth_manifolds}, using covering maps introduced in \cite{LiftingAction,MontaldiOrtega2009}, we construct liftings of smooth group actions. In Section~\ref{Chapter_Embedding_spherical_quandles_in_Lie groups}, we construct embeddings of spherical quandles and thereby prove Theorem~\ref{main_theorem}. In Section~\ref{section_preliminaries_Bergman_Akita_embedding}, as preparation for Theorem~\ref{theorem_ASuzuki_S1} and Theorem~\ref{theorem_ASuzuki_S3}, we review Bergman's embedding and Akita's embedding and develop the necessary preliminaries. In Section~\ref{section_SUZUKI_result_S1}, we focus on the abelian Lie group \(SO(2)\) and state Theorem~\ref{theorem_ASuzuki_S1}. In Section~\ref{section_SUZUKI_result_S3}, we focus on the Lie group \(SU(2)\) and state Theorem~\ref{theorem_ASuzuki_S3}.

\begin{rem}
While we were preparing this paper, a work claiming the existence of a counterexample to Question~\ref{conj_embedding_of_smooth_quandles} appeared; see Arai~\cite[Remark 2.6]{Arai2025detecting}, Kai~\cite[\S 3]{Kai2025quandle}, and Arai--Kai~\cite[\S 4.2]{AraiKai2026detecting}.
\end{rem}

\section{Quandle basics}
\label{section_preliminaries}
In this section, we introduce quandles and the algebraic connectivity of quandles. See \cite{Kamadabook,NSK} for more details.

\begin{definition}[\cite{Joyce1982,Matveev1982}]
\label{def_quandle}
A \textit{quandle} is a set $X$ with a binary operation $\triangleright:X\times X\to X$ satisfying the three conditions:

\noindent(Q1) $x\triangleright x = x$ for any $x\in X$.

\noindent(Q2) The map $S_y:X\to X$ defined by $x\mapsto x\triangleright y$ is bijective for any $y\in X$.

\noindent(Q3) $(x\triangleright y)\triangleright z = (x\triangleright z)\triangleright(y\triangleright z)$ for any $x,y,z\in X$.
\end{definition}
We denote $S^{-1}_{y}(x)$ by $x\triangleright^{-1}y$ for $x,y\in X$. The conditions Q1, Q2, and Q3 in Definition \ref{def_quandle} correspond to Reidemeister moves, operations for knot diagrams in knot theory, I, II and III respectively. See Figure \ref{pic_def_quandle_operation}.

\begin{figure}[hbtp]
\centering
\includegraphics[width=13cm]{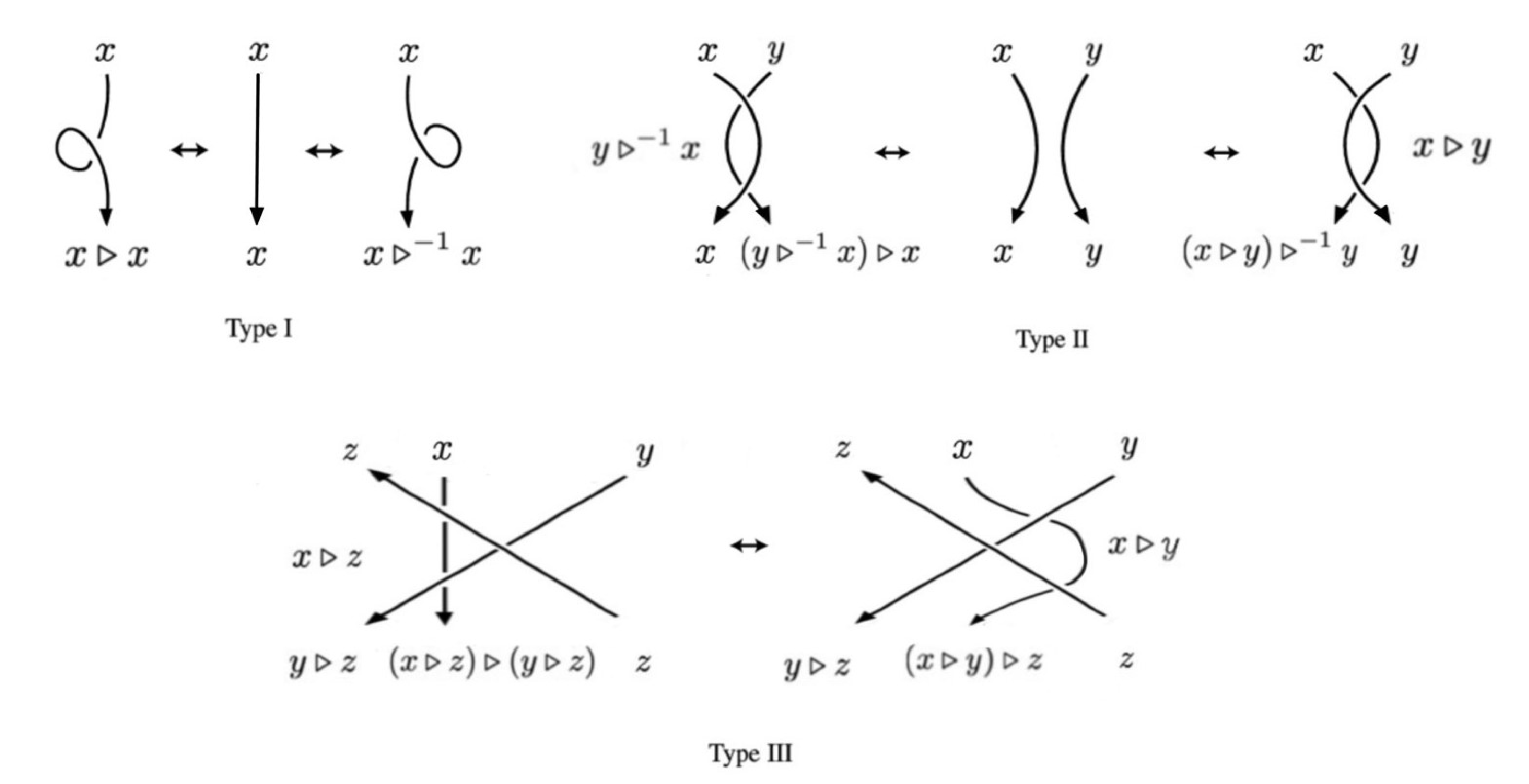}
\caption{Geometric interpretation of the axioms in Definition \ref{def_quandle}.}
\label{pic_def_quandle_operation}
\end{figure}

We define the algebraic connectivity of quandles. Suppose that $X$ and $Y$ are quandles. A map $f:X\to Y$ is called a homomorphism if $f(x\triangleright y)=f(x)\triangleright f(y)$ for any $x,y \in X$. Isomorphisms and automorphisms are defined similarly. We denote the automorphism group of a quandle $X$ by $\operatorname{Aut}X$. By the axiom Q3, the bijection $S_{y}$ in Q2 is an automorphism. Then we call the subgroup of $\operatorname{Aut}X$ generated by $\{S_{y}\}_{y\in X}$ the inner automorphism group and denote it by $\operatorname{Inn}X$. The inner automorphism group $\operatorname{Inn}X$ acts on $X$ on the right naturally. We say that $X$ is algebraically connected if the action of $\operatorname{Inn}X$ on $X$ is transitive.

Let $G$ be a group, and $\triangleright:G\times G\to G$ be a binary operation defined by $g\triangleright h=h^{-1}gh$. Then the algebraic system $(G,\triangleright)$ is a quandle called the conjugation quandle. We denote the quandle by $\operatorname{Conj}G$. 

We recall the spherical quandle introduced by \cite{AzcanFenn1994}. Let $\langle-,-\rangle:\mathbb{R}^{n+1}\times\mathbb{R}^{n+1}\to\mathbb{R}$ be the Euclidean inner product, and let $S^{n}$ be the $n$-sphere,
\[
\left\{
\bm{x}=(x_1,x_2,\dots,x_{n+1})\in \mathbb{R}^{n+1}
\mid
\langle \bm{x},\bm{x}\rangle
=x_1^2+x_2^2+\cdots+x_{n+1}^2=1
\right\}.
\]
We define the binary operation $\triangleright:S^{n}\times S^{n}\to  S^{n}$ as $
\bm{x}\triangleright \bm{y}=2{\langle \bm{x},\bm{y}\rangle}\bm{y}-\bm{x}
$ 
for all $\bm{x},\bm{y}\in S^{n}$. Then $(S^{n},\triangleright)$ is a quandle and is called the spherical quandle $S^{n}_{\mathbb{R}}$.

The inner automorphism $S_{\bm{y}}$ can be interpreted as a linear transformation which identically acts on $\bm{y}$ and $-\operatorname{Id}$ on the subspace orthogonal to $\bm{y}$. Applying the Cartan–Dieudonn\'{e} theorem, we get the following fact.

\begin{proposition}[\cite{Nosaka2017}]
Suppose $n$ is a positive integer. If $n$ is odd, then $\operatorname{Inn}S^{n}_{\mathbb{R}}$ is isomorphic to the orthogonal group $O(n+1)$. If $n$ is even, then $\operatorname{Inn}S^{n}_{\mathbb{R}}$ is isomorphic to the special 
orthogonal group $SO(n+1)$.
\end{proposition}
The action of the inner automorphism group on the spherical quandle $S^{n}_{\mathbb{R}}$ coincides with the natural action of the orthogonal group $O(n+1)$ or the special orthogonal group $SO(n+1)$ on $S^{n}$ and is transitive. Hence we get the following fact.
\begin{proposition}
The spherical quandle is algebraically connected.
\end{proposition}

\section{Action of Lie groups on smooth manifolds}
\label{chapter_Action_of_Lie_groups_on_smooth_manifolds}

\subsection{Smooth actions of Lie groups and the Lie--Palais theorem}
In this section, we discuss smooth actions of Lie groups and the Lie--Palais theorem. See \cite{GorbatsevichOnishchikVinberg1997,Wang2013} for more details.

Let $G$ be a finite-dimensional Lie group, $M$ a smooth manifold and $\mathfrak{X}(M)$ the Lie algebra of the smooth vector fields on $M$. Suppose that $G$ acts smoothly on $M$ on the right. The action induces a group antihomomorphism $\tau:G\to\operatorname{Diff}M$, where $\operatorname{Diff}M$ is the group of diffeomorphisms of $M$. Then $\tau$ induces an infinitesimal action, that is, a Lie algebra antihomomorphism $d\tau:\mathfrak{g}\to \mathfrak{X}(M)$: for each $X\in\frak{g}$, we define a vector field $X_{M}\in\mathfrak{X}(M)$ as
\[
(X_M(x))(\xi)=
\left.\frac{d}{dt}\right\vert_{t=0}
\xi(x\cdot\operatorname{exp}t X),
\]
where $x\in M$ and $\xi$ is a smooth function defined in a neighborhood of the point $x$. The generating vector field \(X_{M}\) associated with \(X\) is complete, since its flow is given by the action of the corresponding one-parameter subgroup.

Conversely, in certain cases, an infinitesimal action can be integrated to a Lie group action. One such case is the Lie–Palais theorem.
\begin{theorem}[Lie--Palais theorem for right actions; cf.~\cite{Palais1957}]
\label{theorem_Lie-Palais}
Let $G$ be a connected and simply connected Lie group, $\phi:\mathfrak{g}\to\mathfrak{X}(M)$ be a Lie algebra antihomomorphism such that each $X_{M}=\phi(X)$ is complete. Then there exists a unique smooth right action $\tau:G\to\operatorname{Diff}M$ such that $d\tau=\phi$.
\end{theorem}

\subsection{Lifting Lie group actions}
\label{section_Lifting_Lie group_actions}
Suppose that a connected Lie group $G$ acts smoothly on a connected manifold $M$, and suppose that $N$ is a covering space of $M$. Then we can construct a smooth action of the universal covering group of $G$ on $N$. We study the concrete construction of the action in this subsection. The construction is introduced in \cite{LiftingAction,MontaldiOrtega2009}. However, we do not describe the original construction, but rather a reformulation of the method sketched in \cite[Remark 2.2]{LiftingAction}.

Let $p:\tilde{G}\to G$ and $\pi_{N}:N\to M$ be covering maps of a connected Lie group $G$ and a connected manifold $M$ respectively. We remark that $p$ induces a Lie algebra isomorphism $p_{*}:\tilde{\mathfrak{g}}\to\mathfrak{g}$, where $\tilde{\mathfrak{g}}$ and $\mathfrak{g}$ are the Lie algebras of $\tilde{G}$ and $G$ respectively.

Suppose $G$ acts on $M$ on the right. Then the action induces a unique group antihomomorphism $\tau:{G}\to\operatorname{Diff}M$. The group antihomomorphism also induces a Lie algebra antihomomorphism $d\tau:\mathfrak{g}\to\mathfrak{X}(M)$. For any $X\in\mathfrak{g}$, an element $\tilde{X}\in\tilde{\mathfrak{g}}$ and a vector field $\tilde{X}_{N}\in \mathfrak{X}(N)$ satisfying
\[
\left\{
\begin{array}{l}
p_{*}(\tilde{X})=X\\
d\pi_{N}\tilde{X}_{N}=X_{M}\circ\pi_{N}
\end{array}
\right.
\]
are uniquely determined. Here, we used Proposition \ref{prop_covering_map_vectorfield}.
\begin{proposition}
\label{prop_covering_map_vectorfield}
Let $\tilde{M}$ and $M$ be smooth manifolds without boundary, $p:\tilde{M}\to M$ be a smooth covering map and $X$ be a smooth vector field on $M$. Then there exists a uniquely defined smooth vector field $\tilde{X}$ on $\tilde{M}$ such that
\[
dp(x)\tilde{X}(x)=X\circ p(x).
\]
\end{proposition}
Then we obtain a Lie algebra antihomomorphism
\[
\tilde{\phi}:
\tilde{\frak{g}}\to\frak{X}(N),
\quad
\tilde{X}\mapsto\tilde{X}_{N}.
\]
Since the flow of the complete vector field \(X_M\) lifts along the covering map \(\pi_N:N\to M\) for all time, the lifted vector field \(\tilde{X}_N\) on \(N\) is also complete. By the Lie--Palais theorem (see Theorem \ref{theorem_Lie-Palais}), $\tilde{\phi}$ induces a group antihomomorphism $\tilde{\tau}:\tilde{G}\to\operatorname{Diff}N$ satisfying $d\tilde{\tau}=\tilde{\phi}$. Hence we are able to construct an action of $\tilde{G}$ on the cover $N$ compatible with an action of $G$ on $M$.

\subsection{The relationship between the lifted Lie group actions and coverings}
Suppose that a connected Lie group $G$ acts on a connected manifold $M$ smoothly, and suppose that $N$ is a covering of $M$. As we saw in Section \ref{section_Lifting_Lie group_actions}, we can construct a smooth action of the universal covering group $\tilde{G}$ of $G$ on the given cover $N$. In this subsection, we see the relationship between the induced action and the covering.
\begin{proposition}[\cite{LiftingAction,MontaldiOrtega2009}]
\label{prop_comm_action_N}
Let $p:\tilde{G}\to G$ be a universal covering map of a connected Lie group $G$. Then the following diagram commutes:
\[
\begin{tikzcd}
  N\times \tilde{G} \ar[r, "\operatorname{action}"] \arrow[d, "\pi_{N}\times p"'] & N \ar[d, "\pi_{N}"] \\
  M\times G \ar[r,"\operatorname{action}"] & M
\end{tikzcd}.
\]
\end{proposition}
Suppose that $\pi:\tilde{M}\to M$ is a universal covering of $M$. The covering map $\pi$ induces a covering morphism $q_{N}$ that makes the following diagram commutative:
\begin{equation*}
\xymatrix{
\tilde{M}\ar[r]^-{q_{N}}\ar[d]_-{\pi}&N\ar[dl]^-{\pi_{N}}\ar@{}@<-2.0ex>[dl]|{}\\
M
}.
\end{equation*}
Applying Proposition \ref{prop_comm_action_N}, we get the following proposition.
\begin{proposition}
\label{prop_action_vs_covering_morphism}
The induced covering morphism $q_{N}$ is $\tilde{G}$-equivariant, that is, for any  $\tilde{x}\in\tilde{M}$ and $\tilde{g}\in\tilde{G}$,
the covering morphism $q_{N}$ satisfies $q_{N}(\tilde{x}\cdot \tilde{g})=q_{N}(\tilde{x})\cdot\tilde{g}$.
\end{proposition}

\section{Embedding spherical quandles in Lie groups}
\label{Chapter_Embedding_spherical_quandles_in_Lie groups}
In this section, we prove Theorem \ref{main_theorem}. One can easily see that spherical quandles can be embedded into the Pin groups in the case $n\geq 2$ since the Spin group is the identity component of the Pin group. Theorem \ref{main_theorem} is already mentioned by Eisermann \cite[Remark 3.12]{Eisermann2014} without proof.

\subsection{Preliminaries}
Let $M(n,\mathbb{R})$ denote the set of all $n$-dimensional real square matrices. We consider two Lie groups: the orthogonal group 
$
O(n)=
\{
g\in M(n,\mathbb{R})
\ :\ {}
{}^{t}gg=I
\}
$ {}, where $I$ is the identity matrix, and the special orthogonal group {}
$
SO(n)=
\{
g\in O(n)
\ :\ {}
\operatorname{det}g=1
\}.
$ {}
Let $h_{n}$ be a diagonal matrix
\[
h_{n}=
\begin{pmatrix}
1 & & &\\
 & -1 & &\\
 & & \ddots &\\
 & & & -1
\end{pmatrix}
\in O(n+1).
\]
If $n$ is even, the diagonal matrix $h_{n}$ is also an element of $SO(n+1)$. The special orthogonal group $SO(n)$ is the identity component of the orthogonal group $O(n)$. The orthogonal group $O(n)$ has two connected components
and has a presentation
\begin{equation}
\label{equation_coset_decomposition_O(n)}
O(n)=
\begin{cases}
SO(n)\sqcup h_{n-1}SO(n), & \text{if } n \text{ is even},\\
SO(n)\sqcup (-I)SO(n), & \text{if } n \text{ is odd},
\end{cases}
\end{equation}
Since $SO(n)$ is a connected Lie group, $SO(n)$ has a unique universal covering group $Spin(n)$ and a universal covering map $p:Spin(n)\to SO(n)$. The map $p$ is a double covering. There are two standard double covers of \(O(n)\), usually denoted by \(Pin^{+}(n)\) and \(Pin^{-}(n)\).
In each case the covering map \(p\colon Pin^{\pm}(n)\to O(n)\) fits into a short exact sequence
\[
1\to \mathbb{Z}/2\mathbb{Z}\to Pin^{\pm}(n)\xrightarrow{p} O(n)\to 1.
\]
For background on \(Pin^{\pm}(n)\), see \cite[\S 3]{AtiyahBottShapiro1964} and \cite{Kirby1990pin}. Throughout this paper, we follow the convention of \cite{Kirby1990pin}: \(Pin^{+}(n)\) is the double cover of \(O(n)\) in which a lift of a reflection has square \(+1\), whereas \(Pin^{-}(n)\) is the double cover in which a lift of a
reflection has square \(-1\). Equivalently, in the corresponding Clifford algebra, a unit vector \(v\) satisfies \(v^{2}=+1\) for \(Pin^{+}(n)\) and \(v^{2}=-1\) for \(Pin^{-}(n)\). In Section~\ref{Chapter_Embedding_spherical_quandles_in_Lie groups}, we treat the two cases simultaneously, since we use only properties common to both: they are double covers of \(O(n)\), and the restriction of either covering to the identity components
is the standard double covering \(p:Spin(n)\to SO(n)\). We distinguish the two double covers explicitly in Section~\ref{section_SUZUKI_result_S3}. We fix a lift \(\tilde{h}_n\in p^{-1}(\{h_n\})\).

Let $\sim_{n}$ be an equivalence relation on the sphere \(S^{n}\) defined by
\[
\bm{x}\sim_{n}\bm{y}
\Longleftrightarrow
\bm{y}=\pm\bm{x}
\quad
(\bm{x},\bm{y}\in S^{n}).
\]
Then $\mathbb{R}P^{n}:=S^{n}/\sim_{n}$ is a real projective space and the natural projection $\pi:S^{n}\to\mathbb{R}P^{n}$ is a covering map. The relation $\sim_{n}$ is a congruence relation: if $\bm{x}_{1},\bm{x}_{2},\bm{y}_{1},\bm{y}_{2}\in S^{n}$ satisfy $\bm{x}_{1}\sim_{n}\bm{x}_{2}$ and $\bm{y}_{1}\sim_{n}\bm{y}_{2}$, then $\bm{x}_{1}\triangleright\bm{y}_{1}\sim_{n}\bm{x}_{2}\triangleright\bm{y}_{2}$. Hence, the quandle structure of $S^{n}_{\mathbb{R}}$ induces a quandle structure on $\mathbb{R}P^{n}$. We denote the quandle over $\mathbb{R}P^{n}$ by $P^{n}_{\mathbb{R}}$. The natural action $S^{n}\curvearrowleft O(n+1)$ induces the action $\mathbb{R}P^{n}\curvearrowleft O(n+1)$ defined by
\[
\pi(\bm{x})\cdot g=\pi(\bm{x}g)
\quad
(\bm{x}\in S^{n},\ {}g\in O(n+1)).
\]
\begin{rem}
The quandle $P^{n}_{\mathbb{R}}$ is introduced by Azcan and Fenn \cite{AzcanFenn1994}. Moreover, the quandle $P^{n}_{\mathbb{R}}$ is faithful.
\end{rem}

Let \(G\) be a Lie group and let \(C_{G}(x)\) be the conjugacy class of \(x\) in \(G\). The natural inclusion \(\operatorname{inc}:C_{G}(x)\to G\) is an immersion since for each $q\in C_{G}(x)$, the differential \(d\operatorname{inc}_q:T_q C_G(x) \to T_q G\) can be identified with the natural inclusion. In particular, if $C_{G}(x)$ is compact, the inclusion \(\operatorname{inc}\) is a smooth embedding. In addition, \(\operatorname{inc}\) is a quandle homomorphism if we regard \(G\) as a conjugation quandle and \(C_G(x)\) as a subquandle. For later reference, we record the preceding observations in the following proposition.

\begin{proposition}
\label{prop_conjugacy-class-inclusion}
Let \(G\) be a Lie group, and let \(C_G(x)=\{g^{-1}xg\mid g\in G\}\) be the conjugacy class of \(x\in G\). Then the natural inclusion \(\operatorname{inc}:C_G(x)\hookrightarrow G\) is a smooth immersion. If, in addition, \(C_G(x)\) is compact, then \(\operatorname{inc}\) is a smooth embedding. Moreover, when \(G\) is regarded as a conjugation quandle and \(C_G(x)\) is viewed as a subquandle, the inclusion \(\operatorname{inc}\) is a quandle homomorphism.
\end{proposition}

\subsection{In the case $n=1$}
\label{section_In-the-case-of-n=1}
The embedding $\iota_{1}$ is presented explicitly:
\[
\iota_{1}:
S^{1}\to O(2),
\quad
(\cos{\theta},\sin{\theta})
\mapsto
\begin{pmatrix}
\cos{\theta} & -\sin{\theta}\\
-\sin{\theta} & -\cos{\theta}
\end{pmatrix}.
\]
We prove $\iota_{1}$ is a smooth embedding and a quandle homomorphism if we regard $S^{1}$ as the one-dimensional spherical quandle and $O(2)$ as a conjugation quandle. It is easy to see that the map $\iota_{1}$ is a smooth embedding since $\iota_{1}$ is the composition of a diffeomorphism 
\[
S^{1}\to SO(2),
\quad
(\cos{\theta} ,\sin{\theta})
\mapsto
\begin{pmatrix}
\cos{\theta} & -\sin{\theta}\\
\sin{\theta} & \cos{\theta}
\end{pmatrix},
\]
a natural embedding $SO(2)\to O(2)$ and a diffeomorphism
\[
O(2)\to O(2),
\quad
g\mapsto
h_{1}g.
\]
Hence it is enough to see that $\iota_{1}$ is a quandle homomorphism. For each $\theta_{1},\theta_{2}\in\mathbb{R}$, {}
$
(\cos{\theta_{1}},\sin{\theta_{1}})\triangleright(\cos{\theta_{2}},\sin{\theta_{2}})
=
(\cos{(2\theta_{2}-\theta_{1})},\sin{(2\theta_{2}-\theta_{1})})
$. Then
\begin{eqnarray*}
\iota_{1}(\cos{\theta_{1}},\sin{\theta_{1}})\triangleright\iota_{1}(\cos{\theta_{2}},\sin{\theta_{2}})
&=&
\begin{pmatrix}
\cos{(2\theta_{2}-\theta_{1})} & -\sin{(2\theta_{2}-\theta_{1})}\\
-\sin{(2\theta_{2}-\theta_{1})} & -
\cos{(2\theta_{2}-\theta_{1})}
\end{pmatrix}\\
&=&\iota_{1}(\cos{(2\theta_{2}-\theta_{1})},\sin{(2\theta_{2}-\theta_{1})})\\
&=&\iota_{1}((\cos{\theta_{1}},\sin{\theta_{1}})\triangleright(\cos{\theta_{2}},\sin{\theta_{2}})).
\end{eqnarray*}

In Section~\ref{section_SUZUKI_result_S1}, we consider the map \(\iota_1\) constructed here.
\subsection{Lifting orthogonal group actions on spheres}
In this section, we use the discussion in Section \ref{section_Lifting_Lie group_actions} to construct the action of the Spin group on the sphere from the action of the special orthogonal group on the real projective space.
\begin{proposition}
\label{LiftActionRPtoS}
The action $\mathbb{R}P^{n}\curvearrowleft SO(n+1)$ induces an action $S^{n}\curvearrowleft Spin(n+1)$ defined by
\begin{equation}
\label{equation_action_S^n-Spin(n+1)}
\bm{x}\cdot\tilde{g}=\bm{x}p(\tilde{g})
\quad
(\bm{x}\in S^{n},\ {}\tilde{g}\in Spin(n+1)).
\end{equation}
\end{proposition}
\begin{proof}
Suppose $\tilde{\tau}:Spin(n+1)\to\operatorname{Diff}S^{n}$ is a group antihomomorphism induced by the action defined by Equation (\ref{equation_action_S^n-Spin(n+1)}). It is enough to show
\[
d\pi\circ d\tilde{\tau}\left(\tilde{X}_{S^{n}}\right)
=
\left(p_{*}\tilde{X}\right)_{\mathbb{R}P^{n}}\circ\pi
\]
for any $\tilde{X}\in\frak{spin}(n+1)$. For any $x\in S^{n}$ and $\xi$ which is a smooth function defined in a neighborhood of $\pi(x)$,  
\begin{eqnarray*}
\left(d\pi\circ d\tilde{\tau}\left(\tilde{X}_{S^{n}}\right)(x)\right)(\xi)
=\left.\frac{d}{dt}\right\vert_{t=0}
\xi\circ\pi(xp(\operatorname{exp}t\tilde{X}))
=(\left(p_{*}\tilde{X}\right)_{\mathbb{R}P^{n}}\circ\pi(x))(\xi).
\end{eqnarray*}
Here, we used the identity \(p(\exp(t\tilde{X}))=\exp(tp_{*}(\tilde{X}))\).
\end{proof}

\subsection{Embedding quandles defined on projective space into Lie groups}
In this section, we show that a natural quandle homomorphism from the spherical quandle to its own inner automorphism group induces an embedding of the real projective space into a Lie group.

Let $n$ be a positive integer, $\operatorname{Conj}(h_{n})$ be a conjugacy class with respect to $h_{n}\in O(n+1)$, and $\operatorname{inn}$ be a map defined by
\begin{equation}\label{equation_inn:Sn-O(n+1)}
\operatorname{inn}:
S^{n}\to O(n+1),
\quad
\bm{e}_{1}g\mapsto g^{-1}h_{n}g.
\end{equation}

\begin{proposition}
\label{prop_welldefinedness_quandle-homness_inn}
The map $\operatorname{inn}$ defined by Equation (\ref{equation_inn:Sn-O(n+1)}) is well-defined and a quandle homomorphism.
\end{proposition}
\begin{proof}
First, we prove that the map $\operatorname{inn}$ is well-defined. The isotropy group $\operatorname{Stab}(n,\bm{e}_{1})\subset O(n+1)$ of the action $S^{n}\curvearrowleft O(n+1)$ with respect to $\bm{e}_{1}=(1,0,\cdots,0)\in S^{n}$ is
\[
\operatorname{Stab}(n,\bm{e}_{1})
=
\left\{
\begin{pmatrix}
1 & 0\\
0 & X
\end{pmatrix}
\in O(n+1)
\ :\ {}
X\in O(n)
\right\}
\cong
O(n).
\]
The isotropy group $\operatorname{Stab}(n,h)\subset O(n+1)$ of the action $\operatorname{Conj}(h_{n})\curvearrowleft O(n+1)$ with respect to $h_{n}\in O(n+1)$ is 
\[
\operatorname{Stab}(n,h)
=
\left\{
\begin{pmatrix}
a & 0\\
0 & X
\end{pmatrix}
\in O(n+1)
\ :\ {}
a\in\{\pm 1\},\ {}X\in O(n)
\right\}.
\]
Hence the map $\operatorname{inn}$ is well-defined since $\operatorname{Stab}(n,\bm{e}_{1})$ is a subset of $\operatorname{Stab}(n,h)$.

Second, we prove that the map $\operatorname{inn}$ is a quandle homomorphism. For any $\bm{x},\bm{y}\in S^{n}$, there exists $A,B\in O(n+1)$ such that $\bm{x}=\bm{e}_{1}A$, $\bm{y}=\bm{e}_{1}B$. It follows that
\begin{equation}
\label{equation_x_tri_y=eAB-1hB}
\bm{x}\triangleright\bm{y}=\bm{e}_{1}AB^{-1}h_{n}B
\end{equation}
since
\begin{eqnarray*}
\bm{x}\triangleright\bm{y}
&=&2\langle\bm{x},\bm{y}\rangle\bm{y}-\bm{x}\\
&=&2\bm{e}_{1}AB^{-1}({}^{t}\bm{e}_{1})\bm{e}_{1}B-\bm{e}_{1}A\\
&=&\bm{e}_{1}AB^{-1}(2({}^{t}\bm{e}_{1})\bm{e}_{1}-I)B\\
&=&\bm{e}_{1}AB^{-1}h_{n}B.
\end{eqnarray*}
Then 
\begin{eqnarray*}
\operatorname{inn}(\bm{x}\triangleright\bm{y})
&=&(AB^{-1}h_{n}B)^{-1}h_{n}(AB^{-1}h_{n}B)\\
&=&(B^{-1}h_{n}B)^{-1}(A^{-1}h_{n}A)(B^{-1}h_{n}B)\\
&=&\operatorname{inn}(\bm{x})\triangleright \operatorname{inn}
(\bm{y}).
\end{eqnarray*}
Hence the map $\operatorname{inn}$ is a quandle homomorphism.
\end{proof}
\begin{proposition}\label{prop_iso_RP_Conj(h)}
The quandle homomorphism $\operatorname{inn}:S^n_{\mathbb{R}}\to\operatorname{Conj}O(n+1)$ induces a quandle isomorphism  $i_{n}:P^{n}_{\mathbb{R}}\to\operatorname{Conj}(h_{n})$.
\end{proposition}
\begin{proof}
 For \(g_{1},g_{2}\in O(n+1)\),
\begin{eqnarray*}
\operatorname{inn}(\bm{e}_{1}g_{1})=\operatorname{inn}(\bm{e}_{1}g_{2})
&\Longleftrightarrow&
g_{1}g_{2}^{-1}\in\operatorname{Stab}(n,h)\\
&\Longleftrightarrow&
\bm{e}_{1}g_{1}\sim_{n}\bm{e}_{1}g_{2}.
\end{eqnarray*}
Therefore, \(\operatorname{inn}\) induces an injection \(\operatorname{inn}^{\prime}:P_{\mathbb{R}}^{n} \to O(n+1)\) such that the following diagram commutes
\begin{equation*}
\xymatrix{
S^{n}_{\mathbb{R}}\ar[r]^-{\operatorname{inn}}\ar[d]_-{\pi}&O(n+1)\ar@{}@<-2.0ex>[dl]|{}\\
P^{n}_{\mathbb{R}}\ar[ur]_-{\operatorname{inn}^{\prime}}
}.
\end{equation*}
By Proposition~\ref{prop_welldefinedness_quandle-homness_inn}, \(\operatorname{inn}'\) is an injective quandle homomorphism. Since the image of $\operatorname{inn}^{\prime}$ coincides with $\operatorname{Conj}(h_{n})$, \(\operatorname{inn}^{\prime}\) induces a quandle isomorphism
\[
i_{n}:P^{n}_{\mathbb{R}}\to\operatorname{Conj}(h_{n}),
\quad
\pi(\bm{e}_{1}g)\mapsto g^{-1}h_{n}g
\]
such that the following diagram commutes
\begin{equation*}
\xymatrix{
P^{n}_{\mathbb{R}}\ar[r]^-{\operatorname{inn}^{\prime}}\ar[d]_-{i_{n}}&O(n+1)\ar@{}@<-2.0ex>[dl]|{}\\
\operatorname{Conj}(h_{n})\ar[ur]_-{\operatorname{inc}}
}.
\end{equation*}
\end{proof}
\begin{rem}\label{rem_diffeo_RP_Conj(h)}
The map $i_{n}$ is a diffeomorphism.
\end{rem}
\begin{rem}
The map \(\operatorname{inn}^{\prime}\) is a smooth embedding and therefore gives an affirmative answer to Question~\ref{conj_embedding_of_smooth_quandles}.
\end{rem}

\subsection{A covering space defined over a conjugacy class of a Lie group}
Let \(n\) be a positive integer. In this section, we construct the universal covering of the conjugacy class \(\operatorname{Conj}(h_{n})\) and we study the lifted action of the conjugacy action $\operatorname{Conj}(h_{n})\curvearrowleft SO(n+1)$ via the universal covering of \(\operatorname{Conj}(h_{n})\). 

Throughout this section, we denote elements of Spin and Pin groups by symbols with a tilde, such as \(\tilde{g}\), and elements of the groups \(O(n)\) and \(SO(n)\) by symbols without a tilde, such as \(g\). This is merely a notational device intended to make it easier to distinguish which Lie group an element belongs to. Note that the notation \(\tilde{g}\) does not necessarily mean that it is a lift of an element \(g\) in \(O(n)\) or \(SO(n)\).

The conjugacy class of $O(n+1)$
\[
\operatorname{Conj}({h}_{n})
=
\{
g^{-1}h_{n}g
\mid{}
g\in O(n+1)
\}
\]
has  the right action of \(SO(n+1)\) defined by conjugation. The action is smooth and is also transitive by Equation~\eqref{equation_coset_decomposition_O(n)}.

We define a conjugacy class of \(Pin^{\pm}(n+1)\) 
\[
\operatorname{Conj}(\tilde{h}_{n})
:=
\begin{cases}
\{
\tilde{g}^{-1}\tilde{h}_{n}\tilde{g}
\mid{}
\tilde{g}\in Pin^{+}(n+1)
\},
& \text{if the domain of }p\text{ is }Pin^{+}(n+1),\\
\{
\tilde{g}^{-1}\tilde{h}_{n}\tilde{g}
\mid{}
\tilde{g}\in Pin^{-}(n+1)
\},
& \text{if the domain of }p\text{ is }Pin^{-}(n+1).
\end{cases}
\]
We define the smooth right action of \(Spin(n+1)\) on \(\operatorname{Conj}(\tilde{h}_{n})\) by conjugation. The action is transitive because of the following lemma.
\begin{lemma}
Let \(n\geq 2\) be an integer, and let
\(p:Pin^{\pm}(n+1)\to O(n+1)\) be the double covering. If \(\tilde{z}\in p^{-1}(\{-I\})\), then
\begin{equation}
\label{equation_decomposition_Pin_by_Spin}
Pin^{\pm}(n+1)=
\begin{cases}
Spin(n+1)\sqcup \tilde{h}_{n}Spin(n+1), & \text{if } n \text{ is odd},\\
Spin(n+1)\sqcup \tilde{z}\,Spin(n+1), & \text{if } n \text{ is even}.
\end{cases}
\end{equation}
Moreover, when \(n\) is even, \(\tilde{h}_{n}\tilde{z}=\tilde{z}\tilde{h}_{n}\).
\end{lemma}
\begin{proof}
By considering the decomposition of \(Pin^{\pm}(n)\) into cosets with respect to the subgroup \(Spin(n)\), we obtain Equation~\eqref{equation_decomposition_Pin_by_Spin}.

Let \(f:Pin^{\pm}(n+1)\to Pin^{\pm}(n+1)\) be the smooth map given by
\[
f(\tilde{x})=\tilde{x}^{-1}\tilde{z}^{-1}\tilde{x}\tilde{z}
\qquad (\tilde{x}\in Pin^{\pm}(n+1)).
\]
Since the image of \(p\circ f\) is the trivial subgroup of \(O(n+1)\), the image of \(f\) is the subgroup of \(\operatorname{Ker}p=p^{-1}(\{I\})\cong\mathbb{Z}/2\mathbb{Z}\). Moreover, since the image \(f(Spin(n+1))\) is topologically connected, the image \(f(Spin(n+1))\) is a trivial subgroup of \(Pin^{\pm}(n+1)\). Hence, when \(n\) is even, we obtain \(\tilde{h}_{n}\tilde{z}=\tilde{z}\tilde{h}_{n}\).
\end{proof}

We construct the universal covering of the conjugacy class \(\operatorname{Conj}(h_{n})\).
\begin{proposition}
\label{prop_UniversalCovering_AConjugacyClass}
Let \(n\geq2\) be an integer. Suppose the map $p:Pin^{\pm}(n+1)\to O(n+1)$ is the double covering of $O(n+1)$ and $\tilde{h}_{n}\in Pin^{\pm}(n+1)$ is a point in the fiber over $h_{n}$ with respect to the covering $p$. Then the map
\[
\pi_{h}:\operatorname{Conj}(\tilde{h}_{n})\to\operatorname{Conj}(h_{n}),
\quad
x\mapsto p(x)
\]
is well-defined and is the universal covering of $\operatorname{Conj}(h_{n})$.
\end{proposition}
\begin{proof}
We denote the Lie algebra of $SO(n+1)$ by $\frak{so}(n+1)$. The differential at the identity of the involution of $Spin(n+1)$
\[
\tilde{\Theta}:Spin(n+1)\to Spin(n+1),
\quad
\tilde{g}\mapsto \tilde{h}_{n}^{-1}\tilde{g}\tilde{h}_{n}
\]
coincides with the involution of $\mathfrak{so}(n+1)$, which is isomorphic to the Lie algebra of $Spin(n+1)$,
\[
\theta:\frak{so}(n+1)\to\frak{so}(n+1),
\quad
X\mapsto h_{n}Xh_{n}(=h_{n}^{-1}Xh_{n}=h_{n}Xh_{n}^{-1}).
\]
The Lie group $Spin(n+1)$ acts transitively on $\operatorname{Conj}(\tilde{h}_{n})$ on the right by conjugation. Then the isotropy group of $\tilde{h}_{n}$ is 
\[
\operatorname{Stab}(\tilde{h}_{n})=
\{
\tilde{g}\in Spin(n+1)
\mid
\tilde{g}\tilde{h}_{n}=\tilde{h}_{n}\tilde{g}
\}.
\]
There exists a natural diffeomorphism
\begin{equation}\label{equation_diffeo_S}
\operatorname{Stab}(\tilde{h}_{n})\backslash Spin(n+1)
\to
\operatorname{Conj}(\tilde{h}_{n}),
\quad
\operatorname{Stab}(\tilde{h}_{n})\tilde{g}
\mapsto
\tilde{g}^{-1}\tilde{h}_{n}\tilde{g}.
\end{equation}
Moreover $\operatorname{Stab}(\tilde{h}_{n})$  coincides with an isotropy group $\tilde{K}$ of the Cartan involution $\tilde{\Theta}:Spin(n+1)\to Spin(n+1)$. By \cite [Theorem 3.4]{Borel1961}, $\tilde{K}=\operatorname{Stab}(\tilde{h}_{n})$ is topologically connected.

Similarly, \(SO(n+1)\) acts transitively on \(\operatorname{Conj}(h_n)\) on the right by conjugation. The isotropy group of ${h}_{n}$ with respect to the action is
\[
\left\{
\begin{pmatrix}
1 & \\
 & X_{+}
\end{pmatrix}
\ {:}\ {}
X_{+}\in SO(n)
\right\}
\sqcup
\left\{
\begin{pmatrix}
-1 & \\
 & X_{-}
\end{pmatrix}
\ {:}\ {}
X_{-}\in O(n)\setminus SO(n)
\right\}
\]
and we denote the isotropy group by $\operatorname{Stab}({h}_{n})$. Then, there exists a natural diffeomorphism
\begin{equation}\label{equation_diffeo_RP}
\operatorname{Stab}(h_{n})\backslash SO(n+1)
\to
\operatorname{Conj}(h_{n}),
\quad
\operatorname{Stab}(h_{n}){g}
\mapsto
{g}^{-1}h_{n}g.
\end{equation}
Moreover the identity component of $\operatorname{Stab}({h}_{n})$ is
\[
\operatorname{Stab}(h_{n})_{0}:=
\left\{
\begin{pmatrix}
1 & \\
 & X_{+}
\end{pmatrix}
\ {:}\ {}
X_{+}\in SO(n)
\right\}.
\]

Based on what has been discussed so far, we construct a universal covering of the conjugacy class
$\operatorname{Conj}(h_n)$. According to \cite[p.52 Proposition 4]{chevalley1947theory}, the map
\[
\operatorname{Stab}({h}_{n})_{0}\backslash SO(n+1)
\to
\operatorname{Stab}({h}_{n})\backslash SO(n+1),
\quad
\operatorname{Stab}(h_{n})_{0}g\mapsto\operatorname{Stab}(h_{n})g
\]
is a covering. In particular, the degree of the covering is two since the cardinality of $\operatorname{Stab}({h}_{n})_{0}\backslash\operatorname{Stab}({h}_{n})$ is two. Since $\tilde{K}=\operatorname{Stab}(\tilde{h}_{n})$ is topologically connected, $Spin(n+1)$ is simply connected, and $\operatorname{Stab}(h_{n})_{0}$ is a topologically connected closed subgroup of $SO(n+1)$, we can use \cite[Proposition 3.6]{helgason1979differential}, and the map
\[
\tilde{K}\backslash Spin(n+1)
\to
\operatorname{Stab}({h}_{n})_{0}\backslash SO(n+1),
\quad
\tilde{K}\tilde{g}\mapsto\operatorname{Stab}(h_{n})_{0}p(\tilde{g})
\]
is a universal covering. Then the composition of these maps
\[
\tilde{K}\backslash Spin(n+1)
\to
\operatorname{Stab}({h}_{n})\backslash SO(n+1),
\quad
\tilde{K}\tilde{g}\mapsto\operatorname{Stab}(h_{n})p(\tilde{g})
\]
is a universal covering
. Using the diffeomorphisms defined by Equation  (\ref{equation_diffeo_S}) and Equation (\ref{equation_diffeo_RP}), the map 
\[
\pi_{h}:\operatorname{Conj}(\tilde{h}_{n})\to\operatorname{Conj}(h_{n}),
\quad
x\mapsto p(x)
\]
is also a universal covering.
\end{proof}

We consider the lift of the conjugacy action $\operatorname{Conj}(h_{n})\curvearrowleft SO(n+1)$ via \(\pi_h\).
\begin{proposition}\label{prop_lifting_conj-conj}
An action
$\operatorname{Conj}(\tilde{h}_{n})\curvearrowleft Spin(n+1)$ induced by the action $\operatorname{Conj}(h_{n})\curvearrowleft SO(n+1)$ defined by conjugation and the universal covering $\pi_{h}:\operatorname{Conj}(\tilde{h}_{n})\to\operatorname{Conj}(h_{n})$ constructed in Proposition \ref{prop_UniversalCovering_AConjugacyClass}  coincides with an action defined by conjugation.
\end{proposition}
\begin{proof}
We use the argument of Section~\ref{section_Lifting_Lie group_actions}. It is enough to prove that 
\[
d\pi_{h}\circ \tilde{X}_{\operatorname{Conj}(\tilde{h}_{n})}
=
\left(p_{*}\tilde{X}\right)_{\operatorname{Conj}(h_{n})}\circ\pi_{h}
\]
for any $\tilde{X}\in\mathfrak{spin}(n+1)$. For any $x\in \operatorname{Conj}(\tilde{h}_{n})$ and $\xi$ which is a smooth function defined in a neighborhood of the point $\pi_{h}(x)$, using the definition of the differential, we get
\begin{eqnarray*}
\left(d\pi_{h}\circ \tilde{X}_{\operatorname{Conj}(\tilde{h}_{n})}(x)\right)(\xi)
&=&\left.\frac{d}{dt}\right\vert_{t=0}
\xi\circ\pi_{h}(\operatorname{exp}(-t\tilde{X})x\operatorname{exp}(t\tilde{X}))\\
&=&(\left(p_{*}\tilde{X}\right)_{\operatorname{Conj}(h_{n})}\circ\pi_{h}(x))(\xi)
\end{eqnarray*}
Here, we used the identity \(p(\exp(t\tilde{X}))=\exp(tp_{*}(\tilde{X}))\).
\end{proof}

\subsection{The proof of Theorem~\ref{main_theorem} for $n\geq 2$}
\label{subsection_proof_main_theorem}
In this subsection, we prove Theorem~\ref{main_theorem} for $n\geq2$, using what we have prepared so far. By Proposition~\ref{prop_UniversalCovering_AConjugacyClass}, the map $\pi_h:\operatorname{Conj}(\tilde{h}_{n})\to\operatorname{Conj}(h_{n})$ is a universal covering. By Proposition~\ref{prop_iso_RP_Conj(h)} and Remark~\ref{rem_diffeo_RP_Conj(h)}, the diffeomorphism $i_{n}$ induces a diffeomorphism $\iota_{n}:S^{n}_{\mathbb{R}}\xrightarrow{\sim}\operatorname{Conj}(\tilde{h}_{n})$ such that the following diagram commutes
\[
  \begin{CD}
     (S^{n}_{\mathbb{R}},\bm{e}_{1}) @>{\iota_{n}}>> (\operatorname{Conj}(\tilde{h}_{n}),\tilde{h}_{n})\\
  @V{\pi}VV    @V{\pi_h}VV \\
    (P^{n}_{\mathbb{R}},\pi(\bm{e}_{1}))    @>{i_{n}}>>  (\operatorname{Conj}(h_{n}),{h}_{n})
  \end{CD}.
\]
Furthermore, by Proposition~\ref{prop_action_vs_covering_morphism}, the diffeomorphism \(\iota_n\) is \(Spin(n+1)\)-equivariant with respect to the actions defined in Propositions~\ref{LiftActionRPtoS} and~\ref{prop_lifting_conj-conj}.

The diffeomorphism $\iota_{n}$ is also a quandle homomorphism.
\begin{proposition}
\label{prop_iota_n_construction}
Let \(n\geq2\) be an integer. The map $\iota_{n}$ is a quandle homomorphism.
\end{proposition}
\begin{proof}For any $\bm{x},\bm{y}\in S^{n}$, there exist $A,B\in SO(n+1)$ such that $\bm{x}=\bm{e}_{1}A$, $\bm{y}=\bm{e}_{1}B$. We divide the proof into two cases according to whether \(n\) is even or odd.

\noindent\textbf{Case 1: \(n\) is even.} Since \(\det(B^{-1}h_n B)=1\), we have \(B^{-1}h_n B \in SO(n+1)\). Hence, \(\iota_{n}(\bm{y}) \in Spin(n+1)\). Then, by Proposition \ref{prop_action_vs_covering_morphism} and Equation~\eqref{equation_x_tri_y=eAB-1hB} in the proof of Proposition~\ref{prop_welldefinedness_quandle-homness_inn}, we obtain
\begin{eqnarray*}
\iota_{n}(\bm{x}\triangleright\bm{y})
&=&\iota_{n}(\bm{e}_{1}AB^{-1}h_{n}B)\\
&=&\iota_{n}(\bm{e}_{1}A\cdot\iota_{n}(\bm{y}))\\
&=&\iota_{n}(\bm{x})\cdot \iota_{n}(\bm{y})\\
&=&\iota_{n}(\bm{y})^{-1}\iota_{n}(\bm{x})\iota_{n}(\bm{y})\\
&=&\iota_{n}(\bm{x})\triangleright\iota_{n}(\bm{y}).
\end{eqnarray*}

\noindent\textbf{Case 2: \(n\) is odd.} Since \(\det(A^{-1}h_{n}AB^{-1}h_{n}B)=1\), we have \(A^{-1}h_{n}AB^{-1}h_n B \in SO(n+1)\). Hence, \(\iota_{n}(\bm{x})\iota_{n}(\bm{y}) \in Spin(n+1)\). Then, by Proposition~\ref{prop_action_vs_covering_morphism} and Equation~\eqref{equation_x_tri_y=eAB-1hB} in the proof of Proposition~\ref{prop_welldefinedness_quandle-homness_inn}, we obtain
\begin{eqnarray*}
\iota_{n}(\bm{x}\triangleright\bm{y})
&=&\iota_{n}(\bm{e}_{1}AB^{-1}h_{n}B)\\
&=&\iota_{n}(\bm{e}_{1}AA^{-1}h_{n}AB^{-1}h_{n}B)\\
&=&\iota_{n}(\bm{e}_{1}A\cdot(\iota_{n}(\bm{x})\iota_{n}(\bm{y})))\\
&=&\iota_{n}(\bm{x})\cdot (\iota_{n}(\bm{x})\iota_{n}(\bm{y}))\\
&=&(\iota_{n}(\bm{x})\iota_{n}(\bm{y}))^{-1}\iota_{n}(\bm{x})(\iota_{n}(\bm{x})\iota_{n}(\bm{y}))\\
&=&\iota_{n}(\bm{y})^{-1}\iota_{n}(\bm{x})\iota_{n}(\bm{y})\\
&=&\iota_{n}(\bm{x})\triangleright\iota_{n}(\bm{y}).
\end{eqnarray*}
Therefore, the diffeomorphism $\iota_{n}$ is a quandle homomorphism.
\end{proof}

When \(n\) is even, by Proposition~\ref{prop_conjugacy-class-inclusion}, the natural inclusion map \(\operatorname{inc}:\operatorname{Conj}(\tilde{h}_{n})\to Spin(n+1)\) is both a smooth embedding and a quandle homomorphism when we regard \(Spin(n+1)\) as a conjugation quandle. Then the composition \(\operatorname{inc}\circ\iota_{n}:S^{n}\to Spin(n+1)\) is a smooth embedding and a quandle homomorphism. By a slight abuse of notation, we denote the map \(\operatorname{inc} \circ \iota_n\) simply by \(\iota_n\). Then we obtain Theorem~\ref{main_theorem} for even \(n\).

When \(n\) is odd and \(n\geq3\), the natural inclusion \(\operatorname{inc}:\operatorname{Conj}(\tilde{h}_{n})\to Pin^{\pm}(n+1)\) is also a smooth embedding and a quandle homomorphism by the same argument. By a slight abuse of notation, we denote the smooth embedding \(\operatorname{inc} \circ \iota_n:S^{n}\to Pin^{\pm}(n+1)\) simply by \(\iota_n\). Then we obtain Theorem~\ref{main_theorem} for odd \(n\geq 3\). In Section~\ref{section_SUZUKI_result_S3}, we consider the map \(\iota_3\) constructed here.
\begin{rem}
\label{rem_concreat_presentation_iota_n}
Let \(n\geq 2\) be an integer. The embedding \(\iota_n\) can be described explicitly as follows, using the action defined by Equation~\eqref{equation_action_S^n-Spin(n+1)} in Proposition~\ref{LiftActionRPtoS}:
\[
\iota_{n}(\bm{e}_{1}\cdot\tilde{g})=\tilde{g}^{-1}\tilde{h}_{n}\tilde{g}
\quad
\left(
\tilde{g}\in Spin(n+1)
\right).
\]
In particular, when \(n\geq 3\) is odd, we construct two embeddings: one into \(Pin^{+}(n+1)\) and the other into \(Pin^{-}(n+1)\). To distinguish them, we sometimes denote these embeddings by \(\iota_{n}^{+}\) and \(\iota_{n}^{-}\), respectively.
\end{rem}

\section{Bergman's embedding and Akita's embedding}
\label{section_preliminaries_Bergman_Akita_embedding}
Throughout this section, let \(G\) be a group. In the first half of this section, we present two embeddings of quandles associated with \(G\). More precisely, we describe Bergman's embedding \cite{Bergman2021core} of the core quandle \(\operatorname{Core} G\), and Akita's embedding \cite{Akita2022embedding} of the twisted conjugation quandle \(\operatorname{Conj}(G,\psi)\) determined by an automorphism \(\psi \in \operatorname{Aut} G\). In the second half of this section, we state an observation that holds when \(G\) is abelian and $\psi$ is the inversion automorphism.

\subsection{Construction of embeddings}
We first construct Bergman's embedding of the core quandle into a group. Consider the unit group \(\mathbb{Z}^{\times}=\{\pm 1\}\), and define a group homomorphism \(\operatorname{Sw}:\mathbb{Z}^{\times}\to \operatorname{Aut}(G\times G)\) by
\begin{eqnarray*}
1&\mapsto& [(g,h)\mapsto(g,h)]=\operatorname{id}_{G\times G},\\
-1&\mapsto& [(g,h)\mapsto (h,g)].
\end{eqnarray*}
In what follows, we write \(\operatorname{Sw}_{1}:=\operatorname{Sw}(1)\) and \(\operatorname{Sw}_{-1}:=\operatorname{Sw}(-1)\). We endow the set \(G\times G\times \mathbb{Z}^{\times}\) with the structure of the (external) semidirect product determined by \(\operatorname{Sw}\). Namely, define a binary operation \((G\times G\times\mathbb{Z}^{\times})\times (G\times G\times\mathbb{Z}^{\times})\to G\times G\times\mathbb{Z}^{\times}\) by
\[
(g_1,h_1,a)\cdot(g_2,h_2,b)
:=
\big(\operatorname{Sw}_{b}(g_{1},h_{1})\cdot (g_{2},h_{2}),\,ab\big),
\quad
(g_{1},g_{2},h_{1},h_{2}\in G,\ a,b\in\mathbb{Z}^{\times}),
\]
where the product \(\operatorname{Sw}_{b}(g_{1},h_{1})\cdot (g_{2},h_{2})\) is taken in the direct product \(G\times G\). We denote the resulting group by \((G\times G)\rtimes_{\operatorname{Sw}}\mathbb{Z}^{\times}\). When we later replace \(G\) by a specific group, we will continue to use the same symbol \(\operatorname{Sw}\) by a slight abuse of notation. Let $\triangleright:G\times G\to G$ be a binary operation defined by $g\triangleright h=hg^{-1}h$. Then the algebraic system $(G,\triangleright)$ is a quandle called the core quandle. We denote the quandle by $\operatorname{Core}G$.
\begin{theorem}[Bergman \cite{Bergman2021core}]
The map \(f_{\operatorname{B}}:\operatorname{Core} G \to \operatorname{Conj}\big((G\times G)\rtimes_{\operatorname{Sw}}\mathbb{Z}^{\times}\big)\) defined by
\[
f_{\operatorname{B}}(g):=(g,g^{-1},-1)
\quad
(g\in G)
\]
is a quandle embedding, that is, an injective quandle homomorphism.
\end{theorem}

Next we construct Akita's embedding. Let \(\psi\in \operatorname{Aut}G\) be an automorphism of \(G\). We endow the set \(G\times \mathbb{Z}\) with the structure of the (external) semidirect product determined by \(\psi\). Namely, define a binary operation \((G\times \mathbb{Z})\times (G\times \mathbb{Z})\to G\times \mathbb{Z}\) by
\[
(g,m)\cdot (h,n):=(\psi^{n}(g)h,\, m+n),
\quad (g,h\in G,\ m,n\in \mathbb{Z}).
\]
We denote the resulting group by \(G\rtimes_{\psi}\mathbb{Z}\). When we later replace \(G\) with a specific group, we will continue to use the same symbol \(\psi\) with a slight abuse of notation. We define the twisted conjugation quandle associated with \(\psi\) by the binary operation \(\triangleright:G\times G\to G\) by
\[
g\triangleright h := \psi(h^{-1}g)\,h,\quad(g,h\in G).
\]
Then \((G,\triangleright)\) is a quandle. We denote it by \(\operatorname{Conj}(G,\psi)\).

\begin{theorem}[Akita \cite{Akita2022embedding}]
The map \(f_{\operatorname{A}}:\operatorname{Conj}(G,\psi)\to\operatorname{Conj}(G\rtimes_{\psi}\mathbb{Z})\) defined by
\[
f_{\operatorname{A}}(g):=(g,1)
\quad
(g\in G)
\]
is a quandle embedding, that is, an injective quandle homomorphism.
\end{theorem}

For the embeddings \(f_{\mathrm{B}}\) and \(f_{\mathrm{A}}\), we will continue to use the same notation even when \(G\) is replaced by a specific group by a slight abuse of notation.

\subsection{Coincidence in the Abelian case}
In this subsection, we state results that hold when the group \(G\) is abelian.

In general, given an automorphism \(\psi\in \operatorname{Aut}G\), define a binary operation \(\triangleright:G\times G\to G\) by
\[
g\triangleright h := \psi(gh^{-1})\,h,\quad(g,h\in G).
\]
Then \((G,\triangleright)\) is a quandle. This quandle is called the generalized Alexander quandle, and we denote it by \(\operatorname{Alex}(G,\psi)\).

When \(G\) is abelian, several quandles introduced above coincide.
\begin{proposition}
\label{prop_abelian_coincidence_Core_ConjAkita}
Assume that \(G\) is an abelian group. Let \(\operatorname{Inv}\in \operatorname{Aut}G\) be the inversion automorphism, \(\operatorname{Inv}(g)=g^{-1}\) for \(g\in G\). Then the twisted conjugation quandle \(\operatorname{Conj}(G,\operatorname{Inv})\), the generalized Alexander quandle \(\operatorname{Alex}(G,\operatorname{Inv})\), and the core quandle \(\operatorname{Core}(G)\) have the same quandle operation and hence coincide.
\end{proposition}

Define a group homomorphism \(\xi:\mathbb{Z}\to \mathbb{Z}^{\times}\) by
\[
\xi(n)=(-1)^n,\quad(n\in \mathbb{Z}).
\]
Then we obtain a short exact sequence
$
1
\to2\mathbb{Z}
\to\mathbb{Z}
\xrightarrow{\xi}\mathbb{Z}^{\times}
\to1
$.

Suppose that the automorphism \(\psi\in \operatorname{Aut}G\) is involutive, that is, \(\psi^{2}=\operatorname{id}_{G}\). Then, for every integer \(n\), we have
\[
\psi^{n}=
\begin{cases}
\operatorname{id}_{G} & \text{if \(n\) is even},\\
\psi & \text{if \(n\) is odd}.
\end{cases}
\]
In view of this, we can define a well-defined binary operation \((G\times \mathbb{Z}/2\mathbb{Z})\times (G\times \mathbb{Z}/2\mathbb{Z})\to G\times \mathbb{Z}/2\mathbb{Z}\) by
\[
(g,m+2\mathbb{Z})\cdot(h,n+2\mathbb{Z})
:=
\big(\psi^{n}(g)\,h,\; m+n+2\mathbb{Z}\big),
\quad(g,h\in G,\ m,n\in \mathbb{Z}).
\]
This turns \(G\times \mathbb{Z}/2\mathbb{Z}\) into a group. For convenience, we denote this group by \(G\rtimes_{\psi}\mathbb{Z}/2\mathbb{Z}\). As before, when we later replace \(G\) by a specific group, we will continue to use the same symbol \(\psi\) by a slight abuse of notation.

\begin{proposition}
Let \(G\) be an abelian group, and let \(\operatorname{Inv}\in \operatorname{Aut}G\) be the involutive automorphism given by inversion, \(\operatorname{Inv}(g):=g^{-1}\) for \(g\in G\). Then one can define an injective group homomorphism
\(\iota_G: G\rtimes_{\operatorname{Inv}}\mathbb{Z}/2\mathbb{Z}\to (G\times G)\rtimes_{\operatorname{Sw}}\mathbb{Z}^{\times}\) by
\[
\iota_{G}(g,m+2\mathbb{Z}) := (g,g^{-1},\xi(m)),
\quad
(g\in G,\ m\in\mathbb{Z}),
\]
and this is well defined.
\end{proposition}
\begin{proof}
That \(\iota_G\) is well defined and injective follows from the fact that, for each integer \(m\),
\[
\xi(m)=
\begin{cases}
1 & \text{if \(m\) is even},\\
-1 & \text{if \(m\) is odd}.
\end{cases}
\]

We next show that \(\iota_G\) is a group homomorphism. Let \(g,h\in G\) and \(m,n\in\mathbb{Z}\). Then
\begin{eqnarray*}
\iota_G(g,m+2\mathbb{Z})\cdot \iota_G(h,n+2\mathbb{Z})
&=&
\big(\operatorname{Sw}_{\xi(n)}(g,g^{-1})\cdot (h,h^{-1}),\,\xi(m+n)\big)\\
&=&
\begin{cases}
(gh,\, g^{-1}h^{-1},\, \xi(m+n)) & \text{if \(n\) is even},\\[2mm]
(g^{-1}h,\, gh^{-1},\, \xi(m+n)) & \text{if \(n\) is odd}.
\end{cases}
\end{eqnarray*}
On the other hand,
\begin{eqnarray*}
\iota_G\big((g,m+2\mathbb{Z})\cdot(h,n+2\mathbb{Z})\big)
&=&
(\operatorname{Inv}^{n}(g)h,\;(\operatorname{Inv}^{n}(g)h)^{-1},\;\xi(m+n))\\
&=&
\begin{cases}
(gh,\, h^{-1}g^{-1},\, \xi(m+n)) & \text{if \(n\) is even},\\[2mm]
(g^{-1}h,\, h^{-1}g,\, \xi(m+n)) & \text{if \(n\) is odd}.
\end{cases}
\end{eqnarray*}
Since \(G\) is abelian, these two expressions coincide, and therefore \(\iota_G\) is a group homomorphism.
\end{proof}

At the end of this section, after a brief preparation, we describe the relationship between Bergman's embedding and Akita's embedding when \(G\) is abelian. There is a naturally induced injective quandle homomorphism
$
\operatorname{inj}:\operatorname{Im} f_{\operatorname{A}} \to\operatorname{Conj}\bigl(G\rtimes_{\psi}\mathbb{Z}/2\mathbb{Z}\bigr)
$
making the following diagram commute.
\[
 \xymatrix@!C{
 \operatorname{Conj}(G,\psi)\ar[r]^-{f_{\operatorname{A}}}\ar[d]^-{\operatorname{im}f_{\operatorname{A}}} &\operatorname{Conj}(G\rtimes_{\psi}\mathbb{Z}) \ar[d]^-{\operatorname{id}\times\operatorname{proj}}\\
 \operatorname{Im}f_{\operatorname{A}} \ar[r]^-{\operatorname{inj}} & {\operatorname{Conj}(G\rtimes_{\psi}\mathbb{Z}/2\mathbb{Z})}
}
.\]
Here \(\operatorname{im} f_{\operatorname{A}}\) denotes the canonical surjection induced by \(f_{\operatorname{A}}\), and \(\operatorname{proj}:\mathbb{Z}\to \mathbb{Z}/2\mathbb{Z}\) is the natural projection. In what follows, we sometimes identify \(\operatorname{inj}\circ \operatorname{im} f_{\operatorname{A}}\) with \(f_{\operatorname{A}}\).

\begin{proposition}
\label{prop_Suzuki_theorem_Akitx_Bergmann_abelian}
Assume that \(G\) is abelian and that \(\operatorname{Inv}\in \operatorname{Aut}G\) is the involution given by inversion. Then the following diagram commutes. In particular, \(f_{\mathrm{A}}\) and \(f_{\mathrm{B}}\) can be identified.
\[
 \xymatrix@!C{
 \operatorname{Conj}(G,\operatorname{Inv})\ar[r]^-{f_{\operatorname{A}}}\ar[d]^-{\operatorname{id}_G} &{\operatorname{Conj}(G\rtimes_{\operatorname{Inv}}\mathbb{Z}/2\mathbb{Z})} \ar[d]^-{\iota_G}\\
 \operatorname{Core}G \ar[r]^-{f_{\operatorname{B}}} & {\operatorname{Conj}((G\times G)\rtimes_{\operatorname{Sw}}\mathbb{Z}^{\times})}
}
.\]
\end{proposition}
\begin{proof}
Let \(g\in G\). Then
\[
\iota_G\circ f_{\mathrm{A}}(g)
=\iota_G(g,1)
=(g,g^{-1},\xi(1))
=(g,g^{-1},-1)
=f_{\mathrm{B}}(g).
\]
\end{proof}

\section{Results on the \(1\)-spherical quandle \(S^{1}_{\mathbb{R}}\)}
\label{section_SUZUKI_result_S1}
In this section, we compare the embedding \(\iota_{1}:S^{1}_{\mathbb{R}}\to\operatorname{Conj}O(2)\), constructed in Section~\ref{section_In-the-case-of-n=1}, with Bergman's and Akita's embeddings.

We consider the abelian Lie group \(G=SO(2)\) and an automorphism \(\operatorname{Inv}\in \operatorname{Aut}(SO(2))\). Throughout, \(\operatorname{Inv}:SO(2)\to SO(2)\) denotes the inversion automorphism; that is, for each \(\theta\in\mathbb{R}\),
\[
\operatorname{Inv}\!\begin{pmatrix}
\cos\theta & -\sin\theta\\
\sin\theta & \cos\theta
\end{pmatrix}
:=
\begin{pmatrix}
\cos\theta & -\sin\theta\\
\sin\theta & \cos\theta
\end{pmatrix}^{-1}
=
\begin{pmatrix}
\cos\theta & \sin\theta\\
-\sin\theta & \cos\theta
\end{pmatrix}.
\]
Note that \(\operatorname{Core}(SO(2))\) and \(\operatorname{Conj}(SO(2),\operatorname{Inv})\) coincide by Proposition~\ref{prop_abelian_coincidence_Core_ConjAkita}.

\begin{lemma}
Define a map \(\mathcal{I}_{1}:S^{1}_{\mathbb{R}}\to \operatorname{Core}(SO(2))\) by
\[
\mathcal{I}_{1}(\cos\theta,\sin\theta)
:=
\begin{pmatrix}
\cos\theta & -\sin\theta\\
\sin\theta & \cos\theta
\end{pmatrix},
\quad(\theta\in\mathbb{R}).
\]
Then \(\mathcal{I}_{1}\) is a smooth quandle isomorphism.
\end{lemma}
\begin{proof}
It is straightforward to check that \(\mathcal{I}_{1}\) is a diffeomorphism. Thus it suffices to show that it is a quandle homomorphism. Let \(\alpha,\beta\in\mathbb{R}\). Then
\begin{eqnarray*}
\mathcal{I}_{1}
(\cos{\alpha}, \sin{\alpha})\triangleright\mathcal{I}_{1}(\cos{\beta} , \sin{\beta})
&=&
\begin{pmatrix}
\cos{(2\beta-\alpha)} & -\sin{(2\beta-\alpha)}\\
\sin{(2\beta-\alpha)} & \cos{(2\beta-\alpha)}
\end{pmatrix}\\
&=&
\mathcal{I}_{1}
((\cos{\alpha}, \sin{\alpha})\triangleright(\cos{\beta} , \sin{\beta})).
\end{eqnarray*}
\end{proof}

Since the short exact sequence
\[
1\to
SO(2)\to
O(2)\xrightarrow{\operatorname{det}}
\mathbb{Z}^{\times}
\to1
\]
splits, we obtain an isomorphism described in the next proposition.
\begin{proposition}
One can define a map \(\gamma:O(2)\to SO(2)\rtimes_{\operatorname{Inv}}\mathbb{Z}/2\mathbb{Z}\) by
\[
\gamma(g)
:=
(
\operatorname{Inv}^{\frac{1-\operatorname{det}g}{2}}
(g\begin{pmatrix}
    1 & 0\\
    0 & -1
\end{pmatrix}^{\frac{1-\operatorname{det}g}{2}})
, \frac{1-\operatorname{det}g}{2}+2\mathbb{Z}
)
\quad
(g\in O(2))
\]
and this is well defined. Moreover, \(\gamma\) is a smooth group isomorphism.
\end{proposition}
\begin{proof}
Since every element of \(O(2)\) has determinant \(1\) or \(-1\), the map \(\gamma\) is well defined.
Moreover, \(\gamma\) admits a smooth inverse map
\[
\delta(g,a+2\mathbb{Z})
:=
\begin{pmatrix}
1 & 0\\
0 & -1
\end{pmatrix}^{a}g,
\qquad (g\in SO(2),\ a\in\mathbb{Z}),
\]
and hence \(\gamma\) is a diffeomorphism.

It therefore suffices to show that \(\gamma\) is a group homomorphism. Let \(g,h\in O(2)\) and \(J:=\operatorname{diag}(1,-1)\).
We distinguish four cases according to the values of \(\det g\) and \(\det h\).

\noindent\textbf{Case 1: \(\det g=1\) and \(\det h=1\).}
Then
\[
\gamma(g)\gamma(h)
=(g,0+2\mathbb{Z})(h,0+2\mathbb{Z})
=(gh,0+2\mathbb{Z}).
\]
Since \(\det(gh)=1\), we have \(\gamma(gh)=(gh,0+2\mathbb{Z})\), and hence \(\gamma(g)\gamma(h)=\gamma(gh)\).

\noindent\textbf{Case 2: \(\det g=1\) and \(\det h=-1\).}
Using that \(SO(2)\) is abelian, we compute
\begin{align*}
\gamma(g)\gamma(h)
&=(g,0+2\mathbb{Z})\bigl(\operatorname{Inv}(hJ),1+2\mathbb{Z}\bigr)\\
&=(g,0+2\mathbb{Z})\bigl(Jh^{-1},1+2\mathbb{Z}\bigr)\\
&=\bigl(\operatorname{Inv}(g)\,Jh^{-1},1+2\mathbb{Z}\bigr)\\
&=\bigl(g^{-1}Jh^{-1},1+2\mathbb{Z}\bigr)
=\bigl(Jh^{-1}g^{-1},1+2\mathbb{Z}\bigr).
\end{align*}
Since \(\det(gh)=-1\), we have
\[
\gamma(gh)=\bigl(\operatorname{Inv}(ghJ),1+2\mathbb{Z}\bigr)=\bigl(Jh^{-1}g^{-1},1+2\mathbb{Z}\bigr),
\]
and hence \(\gamma(g)\gamma(h)=\gamma(gh)\).

\noindent\textbf{Case 3: \(\det g=-1\) and \(\det h=1\).}
Using that \(SO(2)\) is abelian and that \(x^{-1}=JxJ\) for \(x\in SO(2)\), we obtain
\begin{align*}
\gamma(g)\gamma(h)
&=\bigl(\operatorname{Inv}(gJ),1+2\mathbb{Z}\bigr)(h,0+2\mathbb{Z})\\
&=\bigl(Jg^{-1},1+2\mathbb{Z}\bigr)(h,0+2\mathbb{Z})\\
&=\bigl(Jg^{-1}h,1+2\mathbb{Z}\bigr)
=\bigl(hJg^{-1},1+2\mathbb{Z}\bigr)
=\bigl(Jh^{-1}g^{-1},1+2\mathbb{Z}\bigr).
\end{align*}
Since \(\det(gh)=-1\), the same computation as in Case~2 yields
\[
\gamma(gh)=\bigl(Jh^{-1}g^{-1},1+2\mathbb{Z}\bigr),
\]
and hence \(\gamma(g)\gamma(h)=\gamma(gh)\).

\noindent\textbf{Case 4: \(\det g=-1\) and \(\det h=-1\).}
Since every element of \(O(2)\setminus SO(2)\) has order \(2\), we have \(h^{-1}=h\). Thus
\begin{align*}
\gamma(g)\gamma(h)
&=\bigl(\operatorname{Inv}(gJ),1+2\mathbb{Z}\bigr)\bigl(\operatorname{Inv}(hJ),1+2\mathbb{Z}\bigr)\\
&=\bigl(\operatorname{Inv}^{2}(gJ)\,\operatorname{Inv}(hJ),0+2\mathbb{Z}\bigr)\\
&=\bigl(gJ\cdot Jh^{-1},0+2\mathbb{Z}\bigr)
=\bigl(gh^{-1},0+2\mathbb{Z}\bigr)
=\bigl(gh,0+2\mathbb{Z}\bigr).
\end{align*}
Since \(\det(gh)=1\), we have \(\gamma(gh)=(gh,0+2\mathbb{Z})\), and hence \(\gamma(g)\gamma(h)=\gamma(gh)\).

Therefore \(\gamma\) is a group homomorphism, and hence a smooth group isomorphism.
\end{proof}

We describe the relationship among Bergman's embedding, Akita's embedding, and the embedding \(\iota_{1}\) constructed in Section~\ref{section_In-the-case-of-n=1}.
\begin{theorem}
\label{theorem_ASuzuki_S1}
Let \(G=SO(2)\), and let \(\operatorname{Inv}\in \operatorname{Aut}G\) be the automorphism given by inversion. Then the following diagram commutes. In particular, \(\iota_{1}\), \(f_{\mathrm{A}}\), and \(f_{\mathrm{B}}\) can be identified.
\[
 \xymatrix@!C{
 S^{1}_{\mathbb{R}}\ar[r]^-{\iota_{1}}\ar[d]^-{\mathcal{I}_{1}} & \operatorname{Conj}(O(2)) \ar[d]^-{\gamma}\\
 \operatorname{Conj}(G,\operatorname{Inv})\ar[r]^-{f_{\operatorname{A}}}\ar[d]^-{\operatorname{id}_G} &{\operatorname{Conj}(G\rtimes_{\operatorname{Inv}}\mathbb{Z}/2\mathbb{Z})} \ar[d]^-{\iota_G}\\
 \operatorname{Core}G \ar[r]^-{f_{\operatorname{B}}} & {\operatorname{Conj}((G\times G)\rtimes_{\operatorname{Sw}}\mathbb{Z}^{\times})}
}
\]
\end{theorem}
\begin{proof}
By Proposition~\ref{prop_Suzuki_theorem_Akitx_Bergmann_abelian}, it suffices to show that \(\gamma\circ \iota_{1}=f_{\mathrm{A}}\circ \mathcal{I}_{1}\). For each \(\theta\in\mathbb{R}\), we have
\[
\gamma\circ \iota_{1}(\cos\theta,\sin\theta)
=
\left(
\begin{pmatrix}
\cos\theta & -\sin\theta\\
\sin\theta & \cos\theta
\end{pmatrix},
1
\right)
=
f_{\mathrm{A}}\circ \mathcal{I}_{1}(\cos\theta,\sin\theta).
\]
\end{proof}

\section{Results on the \(3\)-spherical quandle \(S^{3}_{\mathbb{R}}\)}
\label{section_SUZUKI_result_S3}

In this section, we compare the embedding \(\iota_{3}^{\pm}:S^{3}_{\mathbb{R}}\to\operatorname{Conj}Pin^{\pm}(4)\), constructed in Section~\ref{subsection_proof_main_theorem}, with Bergman's embedding.

Throughout this section, we write the imaginary unit as \(i=\sqrt{-1}\). We consider the simply connected real Lie group \(SU(2)\), defined by
\[
SU(2)
=
\left\{
\begin{pmatrix}
x_{1}+x_{2}i & x_{3}+x_{4}i\\
-x_{3}+x_{4}i & x_{1}-x_{2}i
\end{pmatrix}
\ :\ 
(x_1,x_2,x_3,x_4)\in S^3
\right\}.
\]
As a topological space, the Lie group \(SU(2)\) is homeomorphic to the \(3\)-sphere \(S^{3}\). It is also related to the \(3\)-dimensional spherical quandle as follows.

\begin{proposition}
Define a map \(\mathcal{I}_{2}:S^{3}_{\mathbb{R}}\to \operatorname{Core}(SU(2))\) by
\[
\mathcal{I}_{2}(x_{1},x_{2},x_{3},x_{4})
:=
\begin{pmatrix}
x_{1}+x_{2}i & x_{3}+x_{4}i\\
-x_{3}+x_{4}i & x_{1}-x_{2}i
\end{pmatrix},
\quad ((x_{1},x_{2},x_{3},x_{4})\in S^{3}).
\]
Then \(\mathcal{I}_{2}\) is a quandle isomorphism.
\end{proposition}
\begin{proof}
Since it is straightforward to verify that \(\mathcal{I}_{2}\) is a diffeomorphism, it suffices to show that it is a quandle homomorphism. Let \((x_{1},x_{2},x_{3},x_{4}),(y_{1},y_{2},y_{3},y_{4})\in S^{3}\). A direct computation shows that both
\(\mathcal{I}_{2}(x_{1},x_{2},x_{3},x_{4})\triangleright \mathcal{I}_{2}(y_{1},y_{2},y_{3},y_{4})\) and
\(\mathcal{I}_{2}\bigl((x_{1},x_{2},x_{3},x_{4})\triangleright (y_{1},y_{2},y_{3},y_{4})\bigr)\)
are equal to
\[
\begin{pmatrix}
a+bi & c+di\\
-c+di & a-bi
\end{pmatrix},
\]
where
\begin{align*}
a&=(2y_{1}^{2}-1)x_{1}+2x_{2}y_{1}y_{2}+2x_{3}y_{1}y_{3}+2x_{4}y_{1}y_{4},\\
b&=2x_{1}y_{1}y_{2}+(2y_{2}^{2}-1)x_{2}+2x_{3}y_{2}y_{3}+2x_{4}y_{2}y_{4},\\
c&=2x_{1}y_{1}y_{3}+2x_{2}y_{2}y_{3}+(2y_{3}^{2}-1)x_{3}+2x_{4}y_{3}y_{4},\\
d&=2x_{1}y_{1}y_{4}+2x_{2}y_{2}y_{4}+2x_{3}y_{3}y_{4}+(2y_{4}^{2}-1)x_{4}.
\end{align*}
Therefore, \(\mathcal{I}_{2}\) is a quandle homomorphism.
\end{proof}

\begin{rem}
The inversion map \(\operatorname{Inv}:SU(2)\to SU(2)\) is a quandle automorphism of \(\operatorname{Core}(SU(2))\). Consequently, \(\operatorname{Inv}\circ\mathcal{I}_{2}:S^{3}_{\mathbb{R}}\to\operatorname{Core}(SU(2))\) is also a quandle isomorphism.
\end{rem}

After some preparations, we state the main theorem of this section, Theorem~\ref{theorem_ASuzuki_S3}.

We give an explicit description of the universal covering \(p_{4}:Spin(4)\to SO(4)\). It is known that \(Spin(4)\) is isomorphic to \(SU(2)\times SU(2)\); thus we identify \(Spin(4)\) with \(SU(2)\times SU(2)\). Define \(p_{4}:Spin(4)\to SO(4)\) as follows. For \(g,h\in SU(2)\), write
\[
g=
\begin{pmatrix}
x_{1}+x_{2}i & x_{3}+x_{4}i\\
-x_{3}+x_{4}i & x_{1}-x_{2}i
\end{pmatrix},
\qquad
h=
\begin{pmatrix}
y_{1}+y_{2}i & y_{3}+y_{4}i\\
-y_{3}+y_{4}i & y_{1}-y_{2}i
\end{pmatrix},
\]
and set \(p_{4}(g,h)\) to be the matrix
\[
\left(
\begin{smallmatrix}
 x_1 y_1 + x_2 y_2 + x_3 y_3 + x_4 y_4
 & x_1 y_2 - x_2 y_1 - x_3 y_4 + x_4 y_3
 & x_1 y_3 + x_2 y_4 - x_3 y_1 - x_4 y_2
 & x_1 y_4 - x_2 y_3 + x_3 y_2 - x_4 y_1
\\
 -x_1 y_2 + x_2 y_1 - x_3 y_4 + x_4 y_3
 & x_1 y_1 + x_2 y_2 - x_3 y_3 - x_4 y_4
 & -x_1 y_4 + x_2 y_3 + x_3 y_2 - x_4 y_1
 & x_1 y_3 + x_2 y_4 + x_3 y_1 + x_4 y_2
\\
 -x_1 y_3 + x_2 y_4 + x_3 y_1 - x_4 y_2
 & x_1 y_4 + x_2 y_3 + x_3 y_2 + x_4 y_1
 & x_1 y_1 - x_2 y_2 + x_3 y_3 - x_4 y_4
 & -x_1 y_2 - x_2 y_1 + x_3 y_4 + x_4 y_3
\\
 -x_1 y_4 - x_2 y_3 + x_3 y_2 + x_4 y_1
 & -x_1 y_3 + x_2 y_4 - x_3 y_1 + x_4 y_2
 & x_1 y_2 + x_2 y_1 + x_3 y_4 + x_4 y_3
 & x_1 y_1 - x_2 y_2 - x_3 y_3 + x_4 y_4
\end{smallmatrix}
\right).
\]
A direct computation shows that \(p_{4}\) is a group homomorphism. Moreover, writing \(I_{2}\) for the \(2\times 2\) identity matrix, we have
\[
\operatorname{Ker}p_{4}=\{(I_{2},I_{2}),\,(-I_{2},-I_{2})\}\cong \mathbb{Z}/2\mathbb{Z},
\]
and hence \(p_{4}\) is a double covering.

By Bott periodicity, \(Pin^{+}(4)\) and \(Pin^{-}(4)\) are isomorphic as Lie groups, but not as double covers of \(O(4)\). It is known that \(Pin^{+}(4)\) is isomorphic to \(Spin(4)\rtimes_{\operatorname{Sw}}\mathbb{Z}^{\times}\). Thus, we identify the Lie groups;
\begin{equation}
\label{equation_Pin(4)=Spin(4)rtimesZ2Z}
Pin^{+}(4)
=Pin^{-}(4)
= Spin(4)\rtimes_{\operatorname{Sw}}\mathbb{Z}^{\times}.
\end{equation}
Then, recalling that \(h_{3}=\operatorname{diag}(1,-1,-1,-1)\), the double covering maps \(p_{\pm}:Pin^{\pm}(4)\to O(4)\) are given by
\[
p_{+}(g,h,a)
=
\begin{cases}
p_{4}(g,h) & \text{if \(a=1\)},\\
p_{4}(g,h)(-h_{3}) & \text{if \(a=-1\)},
\end{cases}
\quad
(g,h\in SU(2),a\in\mathbb{Z}^{\times}),
\]
and
\[
p_{-}(g,h,a)
=
\begin{cases}
p_{4}(g,h) & \text{if \(a=1\)},\\
p_{4}(g,h)h_{3} & \text{if \(a=-1\)},
\end{cases}
\quad
(g,h\in SU(2),a\in\mathbb{Z}^{\times}).
\]
Define
\[
\tilde{H}_{3,+}:=(I_{2},-I_{2},-1)\in (SU(2)\times SU(2))\rtimes_{\operatorname{Sw}}\mathbb{Z}^{\times}=Pin^{+}(4),
\]
and
\[
\tilde{H}_{3,-}:=(I_{2},I_{2},-1)\in (SU(2)\times SU(2))\rtimes_{\operatorname{Sw}}\mathbb{Z}^{\times}=Pin^{-}(4).
\]
Then \(\tilde{H}_{3,+}\in p^{-1}_{+}(\{h_{3}\})\) and \(\tilde{H}_{3,-}\in p^{-1}_{-}(\{h_{3}\})\). Moreover, the facts that
\[
\tilde{H}_{3,+}^2=(-I_2,-I_2,1),
\quad
\tilde{H}_{3,-}^2=(I_2,I_2,1)
\]
and that \(h_3\) is the product of the three reflections:
\[
h_3=
\operatorname{diag}(1,-1,1,1)\operatorname{diag}(1,1,-1,1)\operatorname{diag}(1,1,1,-1)
\]
are compatible with the construction of \(Pin^{\pm}(4)\). Therefore, by Remark~\ref{rem_concreat_presentation_iota_n}, the maps \(\iota_{3}^{\pm}:S^{3}\to Pin^{\pm}(4)\) can be described explicitly as
\[
\iota_{3}^{+}(\bm{e}_{1}p_{4}(g,h))
=
(h^{-1}g,-g^{-1}h,-1)
\quad (g,h\in SU(2)),
\]
and
\[
\iota_{3}^{-}(\bm{e}_{1}p_{4}(g,h))
=
(h^{-1}g,g^{-1}h,-1)
\quad (g,h\in SU(2)).
\]

Let \(\tau:Pin^{+}(4)\to Pin^{-}(4)\) be the map defined by
\[
\tau(g,h,a):=(g,ah,a)
\quad
\left(
(g,h,a)\in Pin^{+}(4)=(SU(2)\times SU(2))\rtimes_{\operatorname{Sw}}\mathbb{Z}^{\times}
\right).
\]
Under the identification given by Equation~\eqref{equation_Pin(4)=Spin(4)rtimesZ2Z}, \(\tau\) is a diffeomorphism, an involution, and a quandle automorphism of \(\operatorname{Conj}Pin^{+}(4)\). Moreover, since \(\tau\circ\iota_3^+=\iota_3^-\), we may identify \(\iota_3^+\) with \(\iota_3^-\) via \(\tau\). Note that \(\tau\) is not a group homomorphism.

\begin{theorem}
\label{theorem_ASuzuki_S3}
The following diagram commutes. In particular, Bergman's embedding for \(SU(2)\) can be identified with \(\iota_{3}^{\pm}\).
\[
 \xymatrix@!C{
 S^{3}_{\mathbb{R}}\ar[r]^-{\iota_{3}^{+}}\ar@{=}[d] & \operatorname{Conj}Pin^{+}(4) \ar[d]^-{\tau}\\
 S^{3}_{\mathbb{R}}\ar[r]^-{\iota_{3}^{-}}\ar[d]^-{\operatorname{Inv}\circ\mathcal{I}_2} & \operatorname{Conj}Pin^{-}(4) \ar@{=}[d]\\
 \operatorname{Core}(SU(2)) \ar[r]^-{f_{\operatorname{B}}} & \operatorname{Conj}((SU(2)\times SU(2))\rtimes_{\operatorname{Sw}}\mathbb{Z}^{\times}).
}
\]
\end{theorem}
\begin{proof}
Every element of \(S^{3}\) can be written as \(\bm{e}_{1}p_{4}(g,h)\) for suitable \(g,h\in SU(2)\). Write \((x_{1},x_{2},x_{3},x_{4}),(y_{1},y_{2},y_{3},y_{4})\in S^{3}\) and set
\[
g=
\begin{pmatrix}
x_{1}+x_{2}i & x_{3}+x_{4}i\\
-x_{3}+x_{4}i & x_{1}-x_{2}i
\end{pmatrix},
\qquad
h=
\begin{pmatrix}
y_{1}+y_{2}i & y_{3}+y_{4}i\\
-y_{3}+y_{4}i & y_{1}-y_{2}i
\end{pmatrix}.
\]
Let
\begin{align*}
z_{1}&=x_1 y_1 + x_2 y_2 + x_3 y_3 + x_4 y_4,\\
z_{2}&=x_1 y_2 - x_2 y_1 - x_3 y_4 + x_4 y_3,\\
z_{3}&=x_1 y_3 + x_2 y_4 - x_3 y_1 - x_4 y_2,\\
z_{4}&=x_1 y_4 - x_2 y_3 + x_3 y_2 - x_4 y_1.
\end{align*}
Then
\begin{equation*}
\mathcal{I}_{2}(\bm{e}_{1}p_{4}(g,h))
=\mathcal{I}_{2}(z_{1},z_{2},z_{3},z_{4})
=
\begin{pmatrix}
z_{1}+z_{2}i & z_{3}+z_{4}i\\
-z_{3}+z_{4}i & z_{1}-z_{2}i
\end{pmatrix}
= g^{-1}h.
\end{equation*}
Therefore,
\[
f_{\mathrm{B}}\circ \operatorname{Inv}\circ \mathcal{I}_{2}\bigl(\bm{e}_{1}p_{4}(g,h)\bigr)
=(h^{-1}g,\,g^{-1}h,\,-1)
=\iota_{3}^{-}\bigl(\bm{e}_{1}p_{4}(g,h)\bigr)
=\tau\circ\iota_{3}^{+}\bigl(\bm{e}_{1}p_{4}(g,h)\bigr).
\]
and hence the diagram commutes.
\end{proof}

It is natural to ask whether analogues of Theorem~\ref{theorem_ASuzuki_S1} and Theorem~\ref{theorem_ASuzuki_S3} can be obtained in dimensions other than \(1\) and \(3\). More precisely, one may ask whether the smooth embedding constructed in Proposition~\ref{prop_iota_n_construction} agrees with Bergman's embedding. Unfortunately, this seems unlikely in other dimensions, since it is known that spheres admit no Lie group structure except in low dimensions: every Lie group is parallelizable. Moreover, Kervaire \cite{Kervaire1958non} and Bott--Milnor \cite{Bott_Milnor1958} showed that the only parallelizable spheres are \(S^{1}\), \(S^{3}\), and \(S^{7}\). In particular, it is known that \(S^{7}\) carries no Lie group structure. See also \cite[Chapter~V, \S12]{Bredon2013topology}.

We also intended to compare our embedding \(\iota_{3}\) with Akita's, in analogy with Theorem~\ref{theorem_ASuzuki_S1}. However, we do not even know whether \(S^{3}_{\mathbb{R}}\) is isomorphic, as a quandle, to any twisted conjugation quandle on a group. Indeed, when \(G=SU(2)\), the automorphism group \(\operatorname{Aut}G\) coincides with the inner automorphism group \(\operatorname{Inn}G\). In \cite{AyuSuzuki2026masterthesis}, it was verified that for no \(\psi\in \operatorname{Aut}G\) is the twisted conjugation quandle \(\operatorname{Conj}(G,\psi)\) isomorphic, as a quandle, to \(\operatorname{Core}G\cong S^{3}_{\mathbb{R}}\); however, it remains unclear whether such an isomorphism exists for some group. More generally, we pose the following problem.
\begin{question}
\label{question_Core_twisted_conj}
Let \(G\) be a group. Does there exist a group $H$ and a group automorphism \(\psi\in \operatorname{Aut}H\) such that \(\operatorname{Core} G\) is isomorphic, as a quandle, to \(\operatorname{Conj}(H,\psi)\)?
\end{question}
In Question~\ref{question_Core_twisted_conj}, the groups \(G\) and \(H\) must have the same cardinality, but they need not be isomorphic as groups. By Proposition~\ref{prop_abelian_coincidence_Core_ConjAkita}, there exist infinitely many groups for which Question~\ref{question_Core_twisted_conj} has an affirmative answer. We expect that finding and classifying such groups will provide insight into how the noncommutativity of a group influences the associated quandle structures.

\section*{Acknowledgements}
Both authors are grateful to Hajime Fujita, Hiroshi Tamaru, and Ryoya Kai. 

The second author would like to express his gratitude to Hiroyuki Ochiai for many valuable comments and discussions. He is grateful to Osamu Saeki for his helpful comments. He thanks Michiko Yonemura for drawing Fig. \ref{pic_def_quandle_operation}. He was supported by JST SPRING, Grant Number JPMJSP2136.

This work was partly supported by MEXT Promotion of Distinctive Joint Research Center Program JPMXP0723833165 and Osaka Metropolitan University Strategic Research Promotion Project (Development of International Research Hubs).
\bibliography{bibliography}
\bibliographystyle{plain}
\end{document}